\documentclass{amsart}

\usepackage[latin1]{inputenc}
\usepackage{amssymb,amsmath,amsthm}
\usepackage{amsfonts}
\usepackage{paralist}

\usepackage{url}
\usepackage{hyphenat}

\newcommand{\vect}[1]{\ensuremath{\mathbf{#1}}}
\newcommand{\card}[1]{\ensuremath{\lvert{#1}\rvert}}

\newcommand{\IN}{\ensuremath{\mathbb{N}}}

\newcommand{\nset}[1]{\ensuremath{[{#1}]}}
\newcommand{\couples}[1][n]{\ensuremath{\binom{#1}{2}}} 
\newcommand{\ontuples}[1]{\ensuremath{\underline{#1}}} 
\newcommand{\symm}[1]{\ensuremath{\Sigma_{#1}}} 
\newcommand{\pos}[2]{\overset{\text{\makebox[0mm][c]{$\overset{#1}{\downarrow}$}}}{\phantom{\makebox[0mm]{f}}{#2}}}

\DeclareMathOperator{\Inv}{Inv}                  
\DeclareMathOperator{\supp}{supp}                
\DeclareMathOperator{\ofo}{ofo}                  
\DeclareMathOperator{\id}{id}                    

\theoremstyle{plain}
\newtheorem{theorem}{Theorem}[section]
\newtheorem{proposition}[theorem]{Proposition}
\newtheorem{lemma}[theorem]{Lemma}

\theoremstyle{definition}
\newtheorem{definition}[theorem]{Definition}
\newtheorem{example}[theorem]{Example}
\newtheorem{problem}[theorem]{Problem}

\begin{document}
\title{On functions with a unique identification minor}
\author{Erkko Lehtonen}
\address{Centro de \'Algebra da Universidade de Lisboa \\
Avenida Professor Gama Pinto~2 \\
1649-003 Lisbon \\
Portugal
\and
Departamento de Matem\'atica \\
Faculdade de Ci\^encias \\
Universidade de Lisboa \\
1749-016 Lisbon \\
Portugal}
\email{erkko@campus.ul.pt}
\date{\today}
\begin{abstract}
We shed some new light to the problem of characterizing those functions of several arguments that have a unique identification minor. The $2$\hyp{}set\hyp{}transitive functions are known to have this property. We describe another class of functions that have a unique identification minor, namely functions determined by the order of first occurrence. We also present some examples of other kinds of functions with a unique identification minor. These examples have a relatively small arity.
\end{abstract}

\maketitle


\section{Introduction}

This paper is a study of the minor quasi\hyp{}order of functions of several arguments.
A function $f \colon A^n \to B$ is a minor of $g \colon A^m \to B$, if $f$ can be obtained from $g$ by the operations of identification of arguments, permutation of arguments, introduction of inessential arguments, and deletion of inessential arguments.
In the special case when a minor is obtained by the identification of a single pair of arguments, we speak of identification minors.

A function is said to have a unique identification minor if all its identification minors are equivalent to each other.
This is an interesting property of functions that is not fully understood, and this paper focuses on the following open problem.

\begin{problem}
\label{prob:uniqueidm}
Characterize the functions that have a unique identification minor.
\end{problem}

This problem was previously posed, albeit in a different formalism, by Bouaziz, Couceiro and Pouzet \cite[Problem~2(ii)]{BouCouPou} in the context of a study of the join\hyp{}irreducible members of the minor ordering of Boolean functions.
Join\hyp{}irreducibility is a property strictly weaker than that of having a unique identification minor; every function with a unique identification minor is join\hyp{}irreducible but the converse does not hold.

It is well known that the $2$\hyp{}set\hyp{}transitive functions have a unique identification minor (for a proof of this fact, see, e.g., \cite[Proposition~4.3]{Lehtonen-totsymm}; this fact is also implicit in the work of Bouaziz, Couceiro and Pouzet~\cite{BouCouPou}).
In the current paper, we identify another large class of functions that have a unique identification minor: functions determined by the order of first occurrence (see Proposition~\ref{prop:ofominor}). These are functions $f \colon A^n \to B$ that can be decomposed as $f = f^* \circ {\ofo}|_{A^n}$, where $\ofo$ is the mapping that maps each tuple $\vect{a}$ to the list of elements occurring in $\vect{a}$ in the order of first occurrence, with repetitions removed. We will also present examples of functions with a unique identification minor that are, up to equivalence, neither $2$\hyp{}set\hyp{}transitive nor determined by the order of first occurrence (Propositions~\ref{prop:notofonor2settransitive} and~\ref{prop:notofonor2settransitive-modified}). These sporadic examples have arity at most $\card{A} + 1$, and the author conjectures that such examples cannot be found when the arity is strictly greater than $\card{A} + 1$.


\section{Preliminaries}
\label{sec:preliminaries}

The set of positive integers is denoted by $\IN_+$.
For $n \in \IN_+$, the set $\{1, \dots, n\}$ is denoted by $\nset{n}$.
The set of all $2$\hyp{}element subsets of $\nset{n}$ is denoted by $\couples$.
The symmetric group on $\nset{n}$ is denoted by $\symm{n}$.
The identity map on any set is denoted by $\id$.

Let $A$ and $B$ be arbitrary nonempty sets.
A \emph{function} (\emph{of several arguments}) \emph{from $A$ to $B$} is a mapping $f \colon A^n \to B$ for some positive integer $n$, which is called the \emph{arity} of $f$.
In the special case when $A = B$, we speak of \emph{operations on $A$.}

For an $n$\hyp{}tuple $\vect{a} = (a_1, \dots, a_n) \in A^n$ and a map $\tau \colon \nset{m} \to \nset{n}$, we write $\vect{a} \tau$ to mean the $m$\hyp{}tuple $(a_{\tau(1)}, \dots, a_{\tau(m)})$.
Since the $n$\hyp{}tuple $\vect{a}$ is formally a map $\vect{a} \colon \nset{n} \to A$, the $m$\hyp{}tuple $\vect{a} \tau$ is in fact the composite map $\vect{a} \circ \tau$.
A map $\tau \colon \nset{m} \to \nset{n}$ induces a map $\ontuples{\tau} \colon A^n \to A^m$ by the rule $\ontuples{\tau}(\vect{a}) = \vect{a} \tau$ for all $\vect{a} \in A^n$.

A function $f \colon A^n \to B$ is a \emph{minor} of a function $g \colon A^m \to B$ if there exists a map $\tau \colon \nset{m} \to \nset{n}$ such that $f = g \circ \ontuples{\tau}$, i.e., $f(\vect{a}) = g(\vect{a} \sigma)$ for all $\vect{a} \in A^n$.
We shall write $f \leq g$ to mean that $f$ is a minor of $g$.
The minor relation $\leq$ is a quasiorder (a reflexive and transitive relation) on the set of all functions of several arguments from $A$ to $B$, and, as for all quasiorders, it induces an equivalence relation on this set by the following rule: $f \equiv g$ if and only if $f \leq g$ and $g \leq f$. We say that $f$ and $g$ are \emph{equivalent} if $f \equiv g$. Furthermore, $\leq$ induces a partial order on the set of equivalence classes.
Informally speaking, $f$ is a minor of $g$, if $f$ can be obtained from $g$ by permutation of arguments, introduction or deletion of inessential arguments, and identification of arguments. If $f$ and $g$ are equivalent, then each one can be obtained from the other by permutation of arguments and introduction or deletion of inessential arguments.
Note in particular that if $f, g \colon A^n \to B$, then $f \equiv g$ if and only if there exists a permutation $\sigma \in \symm{n}$ such that $f = g \circ \ontuples{\sigma}$.

We are especially interested in those minors that arise when a single pair of arguments is identified.
Let $n \geq 2$, and let $f \colon A^n \to B$. For each $I \in \couples$, we define the function $f_I \colon A^{n-1} \to B$ as $f_I = f \circ \delta_I$, where $\delta_I \colon \nset{n} \to \nset{n - 1}$ is given by
\[
\delta_I(i) =
\begin{cases}
i,      & \text{if $i < \max I$,} \\
\min I, & \text{if $i = \max I$,} \\
i - 1,  & \text{if $i > \max I$,}
\end{cases}
\qquad
\text{for all $i \in \nset{n}$.}
\]
In other words, $f_I(\vect{a}) = f(\vect{a} \delta_I)$ for all $\vect{a} \in A^{n-1}$.
More explicitly, if $I = \{i, j\}$ with $i < j$ and $\vect{a} = (a_1, \dots, a_{n-1}) \in A^{n-1}$, then $\vect{a} \delta_I = (a_1, \dots, a_{j-1}, a_i, a_j, \dots, a_{n-1})$, and we have
\[
f_I(a_1, \dots, a_{n-1}) =
f(a_1, \dots, a_{j-1}, a_i, a_j, \dots, a_{n-1}),
\]
for all $(a_1, \dots, a_{n-1}) \in A^{n-1}$.
Note that $a_i$ occurs twice on the right side of the above equality: both at the $i$\hyp{}th and at the $j$\hyp{}th position. We will refer to the functions $f_I$ ($I \in \couples$) as \emph{identification minors} of $f$.

A function $f \colon A^n \to B$ has a \emph{unique identification minor} if $f_I \equiv f_J$ for all $I, J \in \couples$.

A function $f \colon A^n \to B$ is \emph{invariant} under a permutation $\sigma \in \symm{n}$, if $f = f \circ \ontuples{\sigma}$.
The set of all permutations under which $f$ is invariant constitutes a subgroup of $\symm{n}$, and it is called the \emph{invariance group} of $f$ and denoted by $\Inv f$.
If $\Inv f = \symm{n}$ then $f$ is \emph{totally symmetric.}
A function is \emph{$2$\hyp{}set\hyp{}transitive} if its invariance group is $2$\hyp{}set\hyp{}transitive. Recall that a permutation group $G \leq \symm{n}$ is \emph{$2$\hyp{}set\hyp{}transitive} if it acts transitively on the $2$\hyp{}element subsets of $\nset{n}$, i.e., if for all $i, j, k, \ell \in \nset{n}$ with $i \neq j$ and $k \neq \ell$, there exists $\sigma \in G$ such that $\{\sigma(i), \sigma(j)\} = \{k, \ell\}$.
As mentioned in the introduction, it is well known that the $2$\hyp{}set\hyp{}transitive functions have a unique identification minor.


\section{Functions determined by the order of first occurrence}
\label{sec:ofo}

Let $A$ be a fixed nonempty set. Denote by $A^*$ the set of finite strings over $A$, i.e., the set $\bigcup_{n \geq 0} A^n$ of tuples of all possible lengths.
Furthermore, let $A^n_{\neq}$ be the set of tuples in $A^n$ without repeated elements, i.e., tuples $(a_1, \dots, a_n) \in A^n$ satisfying $a_i \neq a_j$ whenever $i \neq j$.
Clearly, if $n > \card{A}$, then $A^n_{\neq} = \emptyset$.
Denote $A^\sharp := \bigcup_{n \geq 0} A^n_{\neq}$.

Let $\ofo \colon A^* \to A^\sharp$ be the function that maps any tuple $(a_1, \dots, a_n)$ to the tuple obtained from $(a_1, \dots, a_n)$ by removing all duplicates of elements, keeping only the first occurrence of each element occurring in the tuple.
In other words, $\ofo$ maps each tuple $\vect{a}$ to the tuple that lists the different elements occurring in $\vect{a}$ in the order of first occurrence (hence the acronym $\ofo$).

\begin{example}
\begin{align*}
& \ofo(\mathsf{balloon}) = \mathsf{balon} &
& \ofo(\mathsf{kayak}) = \mathsf{kay} \\
& \ofo(\mathsf{motorcycle}) = \mathsf{motrcyle} &
& \ofo(\mathsf{seaplane}) = \mathsf{seapln} \\
& \ofo(\mathsf{sleigh}) = \mathsf{sleigh} &
& \ofo(\mathsf{submarine}) = \mathsf{submarine}
\end{align*}
\end{example}

The function $\ofo$ has remarkable properties.
As noted by Marichal, Teheux, and the current author~\cite{LehMarTeh}, it is an associative string function (i.e., $\ofo(\vect{a} \ofo(\vect{b}) \, \vect{c}) = \ofo(\vect{a} \vect{b} \vect{c})$ for all $\vect{a}, \vect{b}, \vect{c} \in A^*$) and hence also idempotent (i.e., $\ofo(\ofo(\vect{a})) = \ofo(\vect{a})$ for all $\vect{a} \in A^*$).
In order to describe another property, recall that a \emph{left regular band} is a semigroup satisfying the identities $x^2 \approx x$ and $xyx \approx xy$.
In the free left regular band on $A$, the product of elements $\vect{u}$ and $\vect{v}$ is $\ofo(\vect{u} \vect{v})$.
Since $\ofo(\ofo(\vect{u}) \ofo(\vect{v})) = \ofo(\vect{u} \vect{v})$, the function $\ofo$ is a homomorphism of the free semigroup on $A$ to the free left regular band on $A$.

\begin{lemma}
\label{lem:ofodeltaI}
For every $\vect{a} \in A^{n-1}$ and for every $I \in \couples$, it holds that $\ofo(\vect{a}) = \ofo(\vect{a} \delta_I)$.
\end{lemma}

\begin{proof}
The application of $\delta_I$ on the string $\vect{a}$ inserts a repetition of the $(\min I)$\hyp{}th letter of $\vect{a}$ at the $(\max I)$\hyp{}th position. Such an insertion of a repeated letter after its first occurrence (which in this case may be at the $(\min I)$\hyp{}th position or before) has no effect on the image of a string under the function $\ofo$.
\end{proof}

A function $f \colon A^n \to B$ is \emph{determined by the order of first occurrence,} if there exists a map $f^* \colon A^\sharp \to B$ such that $f = f^* \circ \ofo|_{A^n}$.

\begin{proposition}
\label{prop:ofominor}
Let $f^* \colon A^\sharp \to B$, and let $f \colon A^n \to B$.
If $f = f^* \circ {\ofo}|_{A^n}$, then $f_I = f^* \circ {\ofo}|_{A^{n-1}}$ for all $I \in \couples$.
\end{proposition}

\begin{proof}
For any $I \in \couples$ and for any $\vect{a} \in A^{n-1}$, we have, by Lemma~\ref{lem:ofodeltaI}, that
\[
f_I(\vect{a}) =
f(\vect{a} \delta_I) =
f^* \circ \ofo (\vect{a} \delta_I) =
f^* \circ \ofo (\vect{a}).
\]
Therefore $f_I = f^* \circ {\ofo}|_{A^{n-1}}$ for all $I \in \couples$.
\end{proof}

According to Proposition~\ref{prop:ofominor}, the functions determined by the order of first occurrence have a unique identification minor.
In Proposition~\ref{prop:suppord}, we are going to find out how much overlap there is between the class of $2$\hyp{}set\hyp{}transitive functions and the class of functions determined by the order of first occurrence.

We need some notions and tools in order to state and prove Proposition~\ref{prop:suppord}.
Let us first introduce a notational device that will be used many times in the sequel. We write expressions such as
\[
(\dots, \pos{i}{a}, \dots, \pos{j}{b}, \dots)
\qquad \text{or} \qquad
(a_1, \dots, \pos{i}{a}, \dots, \pos{j}{b}, \dots, a_n)
\]
to denote an $n$\hyp{}tuple whose $i$\hyp{}th component is $a$ and the $j$\hyp{}th component is $b$. The remaining components are irrelevant to the argument at hand and they are clear from the context. The indices $i$ and $j$ are always distinct and they may be equal to $1$ or $n$, but it does not necessarily hold that $i < j$; however, if it is known that $i < j$, then we usually write the $i$\hyp{}th component to the left of the $j$\hyp{}th one. Also, whenever possible, we write components indexed by $i$ and $i + 1$ next to each other, and we write components indexed by $1$ or $n$ at the beginning and at the end of the tuple, respectively, as in the following:
\[
(\dots, \pos{i}{a}, \pos{\; i+1}{b}, \dots, \pos{\ell}{c}, \dots, \pos{n}{d}).
\]

Following Berman and Kisielewicz~\cite{BerKis}, we define the mapping $\supp \colon \bigcup_{n \geq 1} A^n \to \mathcal{P}(A)$ by the rule
$\supp(a_1, \dots, a_n) = \{a_1, \dots, a_n\}$.
A function $f \colon A^n \to B$ is \emph{determined by $\supp$,} if there exists a map $f' \colon \mathcal{P}(A) \to B$ such that $f = f' \circ {\supp}|_{A^n}$.

\begin{lemma}[{\cite[Lemma~4.1]{Lehtonen-totsymm}}]
\label{lem:hatsigma}
Let $\sigma \in \symm{n}$ and $I \in \couples$. Then there exists a permutation $\hat{\sigma} \in \symm{n-1}$ that satisfies $\hat{\sigma} \circ \delta_{\sigma^{-1}(I)} = \delta_I \circ \sigma$ and $\hat{\sigma}(\min \sigma^{-1}(I)) = \min I$.
\end{lemma}

\begin{proposition}
\label{prop:suppord}
Assume that $n > \card{A} + 1$, and let $f \colon A^n \to B$. Then the following conditions are equivalent:
\begin{enumerate}[\rm (i)]
\item\label{prop:suppord:ts}
$f$ is totally symmetric and determined by the order of first occurrence.

\item\label{prop:suppord:2st}
$f$ is $2$\hyp{}set\hyp{}transitive and determined by the order of first occurrence.

\item\label{prop:suppord:piIJ}
$f$ is determined by the order of first occurrence and for all $I, J \in \couples$, there exists a bijection $\pi_{IJ} \colon \nset{n-1} \to \nset{n-1}$ such that $\pi_{IJ}(\min J) = \min I$ and
$f(\vect{a} \delta_I) = f(\vect{a} \pi_{IJ} \delta_J)$ for all $\vect{a} \in A^{n-1}$.

\item\label{prop:suppord:supp}
$f$ is determined by $\supp$.
\end{enumerate}
\end{proposition}

\begin{proof}
We will prove the implications $\eqref{prop:suppord:ts} \implies \eqref{prop:suppord:2st} \implies \eqref{prop:suppord:piIJ} \implies \eqref{prop:suppord:ts}$ and $\eqref{prop:suppord:ts} \implies \eqref{prop:suppord:supp} \implies \eqref{prop:suppord:ts}$.

$\eqref{prop:suppord:ts} \implies \eqref{prop:suppord:2st}$:
Total symmetry implies $2$\hyp{}set\hyp{}transitivity.

$\eqref{prop:suppord:2st} \implies \eqref{prop:suppord:piIJ}$:
Assume that $f$ is $2$\hyp{}set\hyp{}transitive, and let $I, J \in \couples$. Then there exists a permutation $\sigma \in \Inv f$ such that $\sigma^{-1}(I) = J$. By Lemma~\ref{lem:hatsigma}, there exists a permutation $\hat{\sigma} \in \symm{n-1}$ such that $\hat{\sigma} \circ \delta_J = \delta_I \circ \sigma$ and $\hat{\sigma}(\min J) = \min I$. Therefore, setting $\pi_{IJ} := \hat{\sigma}$, we have $f(\vect{a} \delta_I) = f(\vect{a} \delta_I \sigma) = f(\vect{a} \pi_{IJ} \delta_J)$ for all $\vect{a} \in A^{n-1}$.

$\eqref{prop:suppord:piIJ} \implies \eqref{prop:suppord:ts}$:
Assume that condition \eqref{prop:suppord:piIJ} holds. Observe first that for all integers $k$ and $\ell$ such that $1 \leq k < \ell \leq n$ and for all $a_1, \dots, a_n \in A$ and for $b, c \in \{a_1, \dots, a_{k-1}, a_{k+1}, \dots, a_{\ell-1}, a_{\ell}, \dots, a_{n-1}\}$, by choosing $I := \{k, \ell\}$ and $J := \{n-1, n\}$, we have that
\begin{equation}
\label{eq:swapbc}
\begin{split}
& f(a_1, \dots, a_{k-1}, b, a_{k+1}, \dots, a_{\ell-1}, b, a_{\ell}, \dots, a_{n-1}) = \\
& f(a_{\pi_{IJ}(1)}, \dots, a_{\pi_{IJ}(n-2)}, b, b) = \\
& f(a_{\pi_{IJ}(1)}, \dots, a_{\pi_{IJ}(n-2)}, c, c) = \\
& f(a_1, \dots, a_{k-1}, c, a_{k+1}, \dots, a_{\ell-1}, c, a_{\ell}, \dots, a_{n-1}),
\end{split}
\end{equation}
where the first and third equalities hold by condition \eqref{prop:suppord:piIJ}, and the second equality holds because $f$ is determined by the order of first occurrence and both $b$ and $c$ occur among $a_{\pi_{IJ}(1)}, \dots, a_{\pi_{IJ}(n-2)}$.

We will show that $f$ is totally symmetric. To this end, it is sufficient to show that $\Inv f$ contains all adjacent transpositions $(m \;\: m + 1)$, $1 \leq m \leq n - 1$, i.e., for every $m \in \nset{n - 1}$,
\begin{equation}
\label{eq:mm+1}
f(a_1, \dots, a_n) =
f(a_1, \dots, a_{m-1}, a_{m+1}, a_m, a_{m+2}, \dots, a_n),
\end{equation}
for all $a_1, \dots, a_n \in A$.

Let $m \in \nset{n-1}$, and let $(a_1, \dots, a_n) \in A^n$. If $a_m = a_{m+1}$, then equality~\eqref{eq:mm+1} obviously holds, so let us assume that $a_m \neq a_{m+1}$; let $x := a_m$, $y := a_{m+1}$. Since $n > \card{A} + 1$, there exist indices $i < j$ and $i' < i$, $j' < j$ such that $\alpha := a_{i'} = a_i$ and $\beta := a_{j'} = a_j$. We need to consider several cases according to the order of elements $i$, $j$ and $m$.
In what follows, we will write
\begin{itemize}
\item ``$\stackrel{\ofo}{=}$'' to indicate that the equality holds because $f$ is determined by the order of first occurrence,
\item ``$\stackrel{pq}{=}$'', where $p, q \in \nset{n}$, to indicate that the equality holds by \eqref{eq:swapbc} for $I = \{p, q\}$.
\end{itemize}

\begin{asparaenum}[\it {Case} 1:]
\item $\{i, j\} \cap \{m, m+1\} = \emptyset$. We only give the details in the case when $i < m$, $m + 1 < j$. The other two cases ($i < j < m$; $m + 1 < i < j$) are proved in a similar way.
\begin{align*}
&
f(\dots, \pos{i}{\alpha}, \dots, \pos{m}{x}, \pos{\;\; m+1}{y}, \dots, \pos{j}{\beta}, \dots)
\stackrel{\ofo}{=}
f(\dots, \pos{i}{\alpha}, \dots, \pos{m}{x}, \pos{\;\; m+1}{y}, \dots, \pos{j}{\alpha}, \dots)
\stackrel{ij}{=} \\ &
f(\dots, \pos{i}{x}, \dots, \pos{m}{x}, \pos{\;\; m+1}{y}, \dots, \pos{j}{x}, \dots)
\stackrel{mj}{=}
f(\dots, \pos{i}{x}, \dots, \pos{m}{y}, \pos{\;\; m+1}{y}, \dots, \pos{j}{y}, \dots)
\stackrel{m+1, j}{=} \\ &
f(\dots, \pos{i}{x}, \dots, \pos{m}{y}, \pos{\;\; m+1}{x}, \dots, \pos{j}{x}, \dots)
\stackrel{ij}{=}
f(\dots, \pos{i}{\alpha}, \dots, \pos{m}{y}, \pos{\;\; m+1}{x}, \dots, \pos{j}{\alpha}, \dots)
\stackrel{\ofo}{=} \\ &
f(\dots, \pos{i}{\alpha}, \dots, \pos{m}{y}, \pos{\;\; m+1}{x}, \dots, \pos{j}{\beta}, \dots)
.
\end{align*}

\item $\{i, j\} \cap \{m, m+1\} \neq \emptyset$. Then $x$ or $y$ occurs before the $m$\hyp{}th position and we clearly have
\[
f(a_1, \dots, \pos{m}{x}, \pos{\;\; m+1}{y}, \dots, a_n)
\stackrel{\ofo}{=}
f(a_1, \dots, \pos{m}{y}, \pos{\;\; m+1}{x}, \dots, a_n)
.
\]
\end{asparaenum}

We conclude that \eqref{eq:mm+1} holds for all $a_1, \dots, a_n \in A$, for any $m \in \nset{n}$, i.e., $\Inv f$ contains all adjacent transpositions $(m \;\: m + 1)$. This implies that $f$ is totally symmetric, as claimed.

$\eqref{prop:suppord:ts} \implies \eqref{prop:suppord:supp}$:
Assume that $f$ is determined by the order of first occurrence. Then there exists $f' \colon A^\sharp \to B$ such that $f = f' \circ {\ofo}|_{A^n}$. Since $f$ is totally symmetric, $f'(\vect{a}) = f(\vect{a} \sigma)$ for any permutation $\sigma$ of $\nset{r}$, for all $\vect{a} \in A^\sharp \cap A^r$, $r \geq 1$. Hence the function $f^* \colon \mathcal{P}(A) \to B$ given by setting $f^*(S) := f'(\vect{a})$, where $\vect{a}$ is any element of $A^\sharp$ such that $\supp(\vect{a}) = S$, is well defined. We have
\begin{multline*}
(f^* \circ \supp) (\vect{a}) =
f^*(\supp(\vect{a})) =
f^*(\supp(\ofo(\vect{a})) = \\
f'(\ofo(\vect{a})) =
(f' \circ \ofo) (\vect{a}) = 
f(\vect{a}),
\end{multline*}
for all $\vect{a} \in A^n$. Thus, $f = f^* \circ {\supp}|_{A^n}$, i.e., $f$ is determined by $\supp$.

$\eqref{prop:suppord:supp} \implies \eqref{prop:suppord:ts}$:
Assume that $f$ is determined by $\supp$. Then $f$ is totally symmetric. Furthermore, $f = f^* \circ {\supp}|_{A^n}$ for some $f^* \colon \mathcal{P}(A) \to B$. Define $f' \colon A^\sharp \to B$ as $f'(a_1, \dots, a_r) := f^*(\{a_1, \dots, a_r\})$, for all $(a_1, \dots, a_r) \in A^\sharp$. Then
\begin{multline*}
(f' \circ \ofo) (\vect{a}) = 
f'(\ofo(\vect{a})) =
f^*(\supp(\ofo(\vect{a})) = \\
f^*(\supp(\vect{a})) =
(f^* \circ \supp) (\vect{a}) =
f(\vect{a}),
\end{multline*}
for all $\vect{a} \in A^n$. Thus, $f = f' \circ {\ofo}|_{A^n}$, i.e., $f$ is determined by the order of first occurrence.
\end{proof}


\section{Other functions with a unique identification minor}
\label{sec:other}

We are now going to show (see Proposition~\ref{prop:notofonor2settransitive}) that if $n = \card{A} + 1$, then there exist functions $f \colon A^n \to B$ such that $f$ has a unique identification minor and $f$ is neither $2$\hyp{}set\hyp{}transitive nor equivalent to any function determined by the order of first occurrence. For this end, we make use of a functional construction presented in~\cite{Lehtonen-totsymm} that provides a function of arity $\card{A} + 1$ with pre\hyp{}specified identification minors.
By forcing all identification minors to be the same, up to equivalence, and by choosing other parameters in a careful way, we end up with a function with the desired properties.

\begin{definition}
\label{def:fGPphi}
Assume that $n = k + 1$ and $A$ is a set such that $\card{A} = k \geq 2$. Let $g' \colon \mathcal{P}(A) \to B$ and let $g \colon A^k \to B$, $g = g' \circ {\supp}|_{A^k}$. Let $G := (g^I)_{I \in \couples}$ be a family of functions $g^I \colon A^k \to B$ satisfying $g^I(\vect{a}) = g(\vect{a})$ whenever $\supp(\vect{a}) \neq A$, and let $P := (\rho_I)_{I \in \couples}$ be a family of permutations in $\symm{k}$. Let $\phi \colon \couples \to \couples$ be a bijection. Define $f_{G,P,\phi} \colon A^n \to B$ by the rule $f_{G,P,\phi}(\vect{b}) = g^{\phi(I)}(\vect{a} \rho_I)$ if $\vect{b} = \vect{a} \delta_I$ for $I \in \couples$.
\end{definition}

The definition of $f_{G,P,\phi}$ is good, because if $\supp(\vect{b}) = A$, then there is a unique $\vect{a} \in A^k$ and a unique $I \in \couples$ such that $\vect{b} = \vect{a} \delta_I$; and if $\supp(\vect{b}) \neq A$, then for every $\vect{a} \in A^k$ and for every $I \in \couples$ satisfying $\vect{b} = \vect{a} \delta_I$, we have $\supp(\vect{a}) = \supp(\vect{b}) \neq A$ and $g^{\phi(I)}(\vect{a} \rho_I) = g'(\supp(\vect{a} \rho_I)) = g'(\supp(\vect{b}))$.
It was shown in~\cite[Lemma~3.13]{Lehtonen-totsymm} that $(f_{G,P,\phi})_I \equiv g^{\phi(I)}$ for every $I \in \couples$.

\begin{proposition}
\label{prop:notofonor2settransitive}
Assume that $n = k + 1$ and $A$ and $B$ are sets such that $\card{A} = k \geq 2$ and $\card{B} \geq 2$.
Then there exist functions $f \colon A^n \to B$ and $f^* \colon A^\sharp \to B$ such that $f_I \equiv f^* \circ {\ofo}|_{A^{n-1}}$ for all $I \in \couples$ but $f$ is not equivalent to any $n$\hyp{}ary function determined by the order of first occurrence. Furthermore, if $k > 2$, then $\Inv f = \{\id\}$, and hence $f$ is not $2$\hyp{}set\hyp{}transitive.
\end{proposition}

\begin{proof}
Let $\alpha$ and $\beta$ be distinct elements of $B$. Define the function $h \colon A^k \to B$ by the rule
\[
h(\vect{a}) =
\begin{cases}
\alpha, & \text{if $\vect{a} = \vect{k}$,} \\
\beta, & \text{otherwise.}
\end{cases}
\]
For $I \in \couples$, let $g^I = h$ and let $\rho_I = (1 \; 2 \; 3 \; \cdots \; k)^i = (i, i+1, \dots, k, 1, \dots, i-1)$, where $i = \min I$. Let $\phi$ be the identity map on $\couples$. Denote $G := (g^I)_{I \in \couples}$, $P := (\rho_I)_{I \in \couples}$. Let $f \colon A^n \to B$ be the function $f_{G,P,\phi}$ as in Definition~\ref{def:fGPphi}.

For $I \in \couples$, let us write $\vect{d}_I := \vect{k} \rho_I \delta_I$.
Then $f(\vect{b}) = \alpha$ if and only if $\vect{b} = \vect{d}_I$ for some $I \in \couples$.
Note that the only element of $A$ with repeated occurrences in $\vect{d}_I$ is $1$, and its occurrences are at the two positions indexed by $I$.

For example, if $k = 4$ and $n = 5$, then 
\begin{align*}
& \vect{d}_{\{1,2\}} = (1, 1, 2, 3, 4), &
& \vect{d}_{\{1,3\}} = (1, 2, 1, 3, 4), &
& \vect{d}_{\{1,4\}} = (1, 2, 3, 1, 4), \\
& \vect{d}_{\{1,5\}} = (1, 2, 3, 4, 1), &
& \vect{d}_{\{2,3\}} = (4, 1, 1, 2, 3), &
& \vect{d}_{\{2,4\}} = (4, 1, 2, 1, 3), \\
& \vect{d}_{\{2,5\}} = (4, 1, 2, 3, 1), &
& \vect{d}_{\{3,4\}} = (3, 4, 1, 1, 2), &
& \vect{d}_{\{3,5\}} = (3, 4, 1, 2, 1), \\
& \vect{d}_{\{3,6\}} = (2, 3, 4, 1, 1). &&&&
\end{align*}
In this case, the function $f$ takes on value $\alpha$ at the points listed above and value $\beta$ elsewhere.

We claim that $f$ is not equivalent to any $n$\hyp{}ary function determined by the order of first occurrence. To see this, suppose on the contrary that $f = f^* \circ \ofo|_{A^n} \circ \ontuples{\sigma}$ for some $f^* \colon A^\sharp \to B$ and $\sigma \in \symm{n}$. Let $I \in \couples$, and let $\vect{c}$ be any tuple that has two occurrences of $2$ and satisfies $\ofo(\vect{c} \sigma) = \ofo(\vect{d}_I \sigma)$. (It is clear that such a tuple exists. Take, for example, $\vect{c} := (u_1, \dots, u_{n-1}, 2) \sigma^{-1}$, where $(u_1, \dots, u_{n-1}) = \ofo(\vect{d}_I \sigma)$.) Then we have
\[
\beta = f(\vect{c}) = f^*(\ofo(\vect{c} \sigma)) = f^*(\ofo(\vect{d}_I \sigma)) = f(\vect{d}_I) = \alpha,
\]
a contradiction.

We also claim that if $k > 2$, then the only permutation under which $f$ is invariant is the identity permutation (and hence, in particular, $f$ is not $2$\hyp{}set\hyp{}transitive). To see this, let $\sigma \in \symm{n}$ and assume that $f = f \circ \ontuples{\sigma}$. Then $\vect{b} \mapsto \vect{b} \sigma$ must map the set $\{\vect{d}_I : I \in \couples\}$ onto itself. Let $J = \{1, 2\}$, and let $K$ be the unique couple in $\couples$ such that $\vect{d}_J \sigma = \vect{d}_K$. Assume that $K = \{p, q\}$ with $p < q$.

Suppose first that $p \geq 2$ and $p + 1 < q < n$. Then
\[
(1, 1, 2, 3, \dots, k-1, k) \sigma = (\dots, \pos{p-1}{k}, \pos{p}{1}, \pos{p+1}{2}, \dots, q-p, \pos{q}{1}, q-p+1, \dots),
\]
and, depending on whether $\sigma(p) = 1$ and $\sigma(q) = 2$, or $\sigma(p) = 2$ and $\sigma(q) = 1$, it holds that $\vect{d}_{\{2,n\}} \sigma = (k, 1, 2, 3, \dots, k-1, 1) \sigma$ equals 
\[
(\dots, \pos{p-1}{1}, \pos{p}{k}, \pos{p+1}{2}, \dots, q-p, \pos{q}{1}, q-p+1, \dots) 
\,\,\,\,\text{or}\,\,\,\,
(\dots, \pos{p-1}{1}, \pos{p}{1}, \pos{p+1}{2}, \dots, q-p, \pos{q}{k}, q-p+1, \dots),
\]
and in both cases we arrive at a contradiction, because neither one of these tuples is of the form $\vect{d}_I$ for some $I \in \couples$.

Suppose then that $p = 2$ and $q = n$. Then
\[
(1, 1, 2, 3, \dots, k-1, k) \sigma = (k, 1, 2, \dots, k-1, 1).
\]
Thus $\sigma$ fixes all elements in $\{3, \dots, n-1\}$ and $\sigma(1) = n$, and we have that either $\sigma = (1 \; n)$ or $\sigma = (1 \; n \; 2)$.
If $n > 4$, then $\vect{d}_{\{3,4\}} \sigma$ is either
\[
(k-1, k, 1, 1, 2, \dots, k-3, k-2) (1 \; n) = (k-2, k, 1, 1, 2, \dots, k-3, k-1)
\]
or
\[
(k-1, k, 1, 1, 2, \dots, k-3, k-2) (1 \; n \; 2) = (k-2, k-1, 1, 1, 2, \dots, k-3, k),
\]
and both possibilities for $\sigma$ give rise to a contradiction, because neither one of these tuples is of the form $\vect{d}_I$ for some $I \in \couples$.
If $n = 4$, then
\begin{gather*}
\vect{d}_{\{2,3\}} (1 \; 4) = (3, 1, 1, 2) (1 \; 4) = (2, 1, 1, 3),
\\
\vect{d}_{\{1,4\}} (1 \; 4 \; 2) = (1, 2, 3, 1) (1 \; 4 \; 2) = (1, 1, 3, 2),
\end{gather*}
and we arrive again at a contradiction.

Suppose then that $p \geq 2$ and $q = p + 1$.
Then
\[
(1, 1, 2, 3, \dots, k-1, k) \sigma = (\dots, \pos{p-1}{k}, \pos{p}{1}, \pos{p+1}{1}, 2, \dots),
\]
and, depending on whether $\sigma(p) = 1$ and $\sigma(p+1) = 2$, or $\sigma(p) = 2$ and $\sigma(p+1) = 2$, it holds that
$\vect{d}_{\{2,n\}} \sigma = (k, 1, 2, 3, \dots, k-1, 1) \sigma$ equals
\[
(\dots, \pos{p-1}{1}, \pos{p}{k}, \pos{p+1}{1}, 2, \dots)
\qquad \text{or} \qquad
(\dots, \pos{p-1}{1}, \pos{p}{1}, \pos{p+1}{k}, 2, \dots),
\]
and in both cases we arrive at a contradiction.

Suppose then that $p = 1$ and $q \geq 4$. Then
\[
(1, 1, 2, 3, \dots, k-1, k) \sigma = (\pos{1}{1}, 2, 3, \dots, q-1, \pos{q}{1}, q, \dots, k),
\]
and, depending on whether $\sigma(1) = 1$ and $\sigma(q) = 2$, or $\sigma(1) = 2$ and $\sigma(q) = 1$, it holds that $\vect{d}_{\{1,3\}} \sigma = (1, 2, 1, 3, \dots, k-1, k) \sigma$ equals
\[
(\pos{1}{1}, 1, 3, \dots, q-1, \pos{q}{2}, q, \dots, k)
\qquad \text{or} \qquad
(\pos{1}{2}, 1, 3, \dots, q-1, \pos{q}{1}, q, \dots, k),
\]
and in both cases we arrive at a contradiction.

Suppose then that $p = 1$ and $q = 3$. Then
\[
(1, 1, 2, 3, \dots, k-1, k) \sigma = (1, 2, 1, 3, 4, \dots, k-1, k),
\]
and, depending on whether $\sigma(1) = 1$ and $\sigma(3) = 2$, or $\sigma(1) = 2$ and $\sigma(3) = 1$, it holds that
$\vect{d}_{\{2,n\}} \sigma = (k, 1, 2, 3, \dots, k-1, 1) \sigma$ equals
\[
(k, 2, 1, 3, 4, \dots, k-1, 1)
\qquad \text{or} \qquad
(1, 2, k, 3, 4, \dots, k-1, 1),
\]
and in both cases we arrive at a contradiction.

Finally, suppose that $p = 1$ and $q = 2$. Then
\[
(1, 1, 2, 3, \dots, k-1, k) \sigma = (1, 1, 2, 3, \dots, k-1, k),
\]
and we have that either $\sigma = (1 \; 2)$ or $\sigma$ is the identity permutation. If $\sigma = (1 \; 2)$, then
\[
\vect{d}_{\{1,3\}} \sigma = (1, 2, 1, 3, \dots, n-1) \sigma = (2, 1, 1, 3, \dots, n-1),
\]
a contradiction.
The only remaining possibility is that $\sigma$ is the identity permutation, and we have arrived at our desired result.
\end{proof}


\section{Functions of small arities}

Let us still consider functions $f \colon A^n \to B$ with $2 \leq n \leq \card{A}$.
It is immediate from the definition that every binary function has a unique identification minor.
Also the $2$\hyp{}set\hyp{}transitive functions and the functions determined by the order of first occurrence have a unique identification minor, regardless of the arity.

In order to explain more easily what follows, let us extend some of the previous notions to partial functions.
An $n$-ary \emph{partial function} from $A$ to $B$ is a map $f \colon S \to B$, where $S \subseteq A^n$.
In the case when $S = A^n$ we speak of \emph{total functions.}
We are mainly interested in partial functions whose domain is $A^n_{=} := A^n \setminus A^n_{\neq}$, i.e., the set of $n$\hyp{}tuples on $A$ with some repeated entries.

In analogy to total functions, a partial function $f \colon S \to B$ with $S \subseteq A^n$ is \emph{invariant} under a permutation $\sigma \in \symm{n}$ if $\ontuples{\sigma}$ maps the domain set $S$ onto itself and $f = f \circ \ontuples{\sigma}|_S$.
Then the notions of \emph{invariance group} and \emph{$2$\hyp{}set\hyp{}transitivity} are defined for partial functions in the same way as for total functions.
Similarly, $f \colon S \to B$ is \emph{determined by the order of first occurrence} if $f = f^* \circ {\ofo}|_S$ for some $f^* \colon A^\sharp \to B$.
We also say that two $n$-ary partial functions $f \colon S \to B$ and $g \colon T \to B$ ($S, T \subseteq A^n$) are \emph{equivalent} if there exists a permutation $\sigma \in \symm{n}$ such that $\ontuples{\sigma}$ maps the set $S$ onto $T$ and $f = g \circ \ontuples{\sigma}|_S$.

Let $f, g \colon A^n \to B$. If $f|_{A^n_{=}} = g|_{A^n_{=}}$, then $f_I = g_I$ for every $I \in \couples$, because the values of $f$ and $g$ in $A^n_{\neq}$ do not play any role in the formation of identification minors.
Consequently, if $n \leq \card{A}$ and $f \colon A^n \to B$ has a unique identification minor (being, for example, $2$\hyp{}set\hyp{}transitive or determined by the order of first occurrence), then by changing the values of $f$ in $A^n_{\neq}$, we can construct other functions that have a unique identification minor but that are not necessarily $2$\hyp{}set\hyp{}transitive or determined by the order of first occurrence.

Thus, if $f \colon A^n \to B$ is a function such that $f|_{A^n_{=}}$ is $2$-set-transitive or determined by the order of first occurrence, then $f$ has a unique identification minor.
It should also be noted that the functional construction presented in Definition~\ref{def:fGPphi} and applied in Proposition~\ref{prop:notofonor2settransitive} can be modified, with obvious changes, to provide examples of functions $f \colon A^n \to B$ ($n \leq \card{A}$) that have a unique identification minor but for which $f|_{A^n_{=}}$ is not, up to equivalence, determined by the order of first occurrence and, in the case that $n \geq 3$, $f|_{A^n_{=}}$ is not $2$\hyp{}set\hyp{}transitive.

\begin{definition}
\label{def:fGPphi-modified}
Let $A$ be a set with $\card{A} = k$, and let $m \in \IN_+$ be such that $2 \leq m \leq k$, and let $n := m + 1$.
Let $g' \colon \mathcal{P}(A) \to B$ and let $g \colon A^m \to B$, $g = g' \circ {\supp}|_{A^m}$.
Let $G := (g^I)_{I \in \couples}$ be a family of functions $g^I \colon A^m \to B$ satisfying $g^I(\vect{a}) = g(\vect{a})$ whenever $\card{\supp{\vect{a}}} < m$, and let $P := (\rho_I)_{I \in \couples}$ be a family of permutations in $\symm{m}$.
Let $\phi \colon \couples \to \couples$ be a bijection.
Let $f_{G,P,\phi} \colon A^n_{=} \to B$ be the partial operation defined by the rule $f_{G,P,\phi}(\vect{b}) = g^{\phi(I)}(\vect{a} \rho_I)$ if $\vect{b} = \vect{a} \delta_I$ for $\vect{a} \in A^m$ and $I \in \couples$.
\end{definition}

The definition of $f_{G,P,\phi}$ is good, because if $\card{\supp(\vect{b})} = m$, then there is a unique $\vect{a} \in A^m$ and a unique $I \in \couples$ such that $\vect{b} = \vect{a} \delta_I$; and if $\card{\supp{\vect{b}}} < m$, then for every $\vect{a} \in A^m$ and for every $I \in \couples$ satisfying $\vect{b} = \vect{a} \delta_I$ we have $\supp(\vect{a}) = \supp(\vect{b})$, so $\card{\supp(\vect{a})} < m$, and we have $g^{\phi(I)}(\vect{a} \rho_I) = g(\vect{a} \rho_I) = g' \circ {\supp}(\vect{a} \rho_I) = g'(\supp(\vect{b}))$.

It is easy to see that if $f \colon A^n \to B$ is any function such that $f|_{A^n_{=}} = f_{G,P,\phi}$, then $f_I \equiv g^{\phi(I)}$ for every $I \in \couples$. Furthermore, Proposition~\ref{prop:notofonor2settransitive} extends to small arities as follows.

\begin{proposition}
\label{prop:notofonor2settransitive-modified}
Let $A$ and $B$ be sets and let $m$ be an integer such that $2 \leq m \leq \card{A}$ and $\card{B} \geq 2$. Let $n := m + 1$.
Then there exist functions $f \colon A^n \to B$ and $f^* \colon A^\sharp \to B$ such that $f_I \equiv f^* \circ {\ofo}|_{A^{n-1}}$ for all $I \in \couples$ but $f|_{A^n_{=}}$ is not equivalent to any $n$\hyp{}ary partial function determined by the order of first occurrence and, in the case that $m \geq 3$, $f|_{A^n_{=}}$ is not $2$\hyp{}set\hyp{}transitive.
\end{proposition}

\begin{proof}
Straightforward modification of the proof of Proposition~\ref{prop:notofonor2settransitive}.
\end{proof}


\section{Concluding remarks}

We have investigated the problem of determining the functions $f \colon A^n \to B$ that have a unique identification minor (see Problem~\ref{prob:uniqueidm}).
While a definitive answer to this problem eludes us, let us summarize here some facts we know.
It is well known from earlier results that the $2$-set-transitive functions have this property, and we showed in this paper (Proposition~\ref{prop:ofominor}) that the functions that are, up to equivalence, determined by the order of first occurrence also have a unique identification minor.
More generally, the functions $f \colon A^n \to B$ such that the restriction $f|_{A^n_{=}}$ is $2$\hyp{}set\hyp{}transitive or equivalent to an $n$\hyp{}ary partial function determined by the order of first occurrence have a unique identification minor. (This generalization is proper only in the case when $n \leq \card{A}$.)
Furthermore, Proposition~\ref{prop:notofonor2settransitive} shows that if $n = \card{A} + 1$, then there exist functions $f \colon A^n \to B$ such that $f$ has a unique identification minor and $f$ is neither $2$\hyp{}set\hyp{}transitive nor equivalent to any function determined by the order of first occurrence.
Analogously, by Proposition~\ref{prop:notofonor2settransitive-modified}, if $n \leq \card{A}$, then there exist functions $f \colon A^n \to B$ such that $f$ has a unique identification minor and $f|_{A^n_{=}}$ is not equivalent to any partial function determined by the order of first occurrence and, in the case that $n \geq 4$, $f|_{A^n_{=}}$ is not $2$\hyp{}set\hyp{}transitive.
Note also that every binary function trivially has a unique identification minor.

A complete and explicit characterization of functions with a unique identification minor remains a topic of further investigation.
In particular, whether there exist functions $f \colon A^n \to B$ with $n > \card{A} + 1$ with a unique identification minor, other than the ones that are, up to equivalence, $2$\hyp{}set\hyp{}transitive or determined by the order of first occurrence, remains an open problem.
The author conjectures that no other such functions exist when $n > \card{A} + 1$.


\section*{Acknowledgments}

The author would like to thank Miguel Couceiro, Maria João Gouveia, Peter Mayr, Karsten Schölzel, and Tamás Waldhauser for inspiring discussions on minors of functions and on the order of first occurrence.

This work was developed within the FCT Project PEst-OE/MAT/UI0143/2014 of CAUL, FCUL.


\begin{thebibliography}{9}
\bibitem{BerKis}
    \textsc{J. Berman,}  \textsc{A. Kisielewicz,}
    On the number of operations in a clone,
    \textit{Proc.\ Amer.\ Math.\ Soc.}\ \textbf{122} (1994) 359--369.

\bibitem{BouCouPou}
    \textsc{M. Bouaziz,}  \textsc{M. Couceiro,}  \textsc{M. Pouzet,}
    Join\hyp{}irreducible Boolean functions,
    \textit{Order} \textbf{27} (2010) 261--282.

\bibitem{Lehtonen-totsymm}
    \textsc{E. Lehtonen,}
    Totally symmetric functions are reconstructible from identification minors,
    \textit{Electron.\ J. Combin.}\ \textbf{21}(2) (2014) \#P2.6.

\bibitem{LehMarTeh}
    \textsc{E. Lehtonen,}  \textsc{J.-L. Marichal,}  \textsc{B. Teheux,}
    Associative string functions,
    \textit{Asian-Eur.\ J. Math.}\ \textbf{7}(4) (2014) 1450059, 18 pp.
\end{thebibliography}
\end{document}